\documentclass[12pt,letterpaper]{article}
\usepackage{natbib}

\usepackage{bbold}
\usepackage{bm}
\usepackage{bbm}

\usepackage{chngcntr}
\usepackage{multirow}

\usepackage[colorlinks=true,urlcolor=blue,citecolor=blue,linkcolor=blue,bookmarks=true,bookmarksopen=false]{hyperref}

\setcitestyle{round,aysep={,},yysep={;}, citesep={;}}

\newcommand\restr[2]{{
  \left.\kern-\nulldelimiterspace 
  #1 
  \vphantom{\big|} 
  \right|_{#2} 
  }}

\newcommand{\vc}{\mathbf{c}}

\newcommand{\ve}{\mathbf{e}}

\newcommand{\vh}{\mathbf{h}}

\newcommand{\vz}{\mathbf{z}}

\newcommand{\vx}{\mathbf{x}}
\newcommand{\vy}{\mathbf{y}}

\newcommand{\vf}{\mathbf{f}}

\newcommand{\vs}{\mathbf{s}}

\newcommand{\vu}{\mathbf{u}}

\renewcommand{\Omega}{\varOmega}

\newcounter{commentcounter}
\long\def\symbolfootnote[#1]#2{\begingroup%
\def\thefootnote{\fnsymbol{footnote}}\footnote[#1]{#2}\endgroup}
\newcommand{\comment}[1]{{\footnotesize\textbf{\textcolor{red}{(C.\arabic{commentcounter})}}\symbolfootnote[4]{\texttt{\textcolor{red}
        {(C.\arabic{commentcounter})~#1}}}}\addtocounter{commentcounter}{1}}

\let\mathbb\undefined  
\usepackage{bbold}

\usepackage{sibarticle}
\usepackage[ruled,linesnumbered]{algorithm2e}

\DeclareSymbolFont{bbold}{U}{bbold}{m}{n} 
\DeclareSymbolFontAlphabet{\mathbbold}{bbold} 

\allowdisplaybreaks

\usepackage{enumitem}

\usepackage{multirow}
\usepackage{booktabs}
\usepackage{rotating}

\newcommand{\exclude}[1]{}
\usepackage{color}

\title{A Polyhedral Study of the  Static Probabilistic Lot-Sizing Problem}
\ShortTitle{Polyhedral Study of the Static Probabilistic Lot-Sizing Problem} \ShortAuthors{Liu and K\"{u}\c{c}\"{u}kyavuz}
\NumberOfAuthors{1}
\FirstAuthor{Xiao Liu, Simge K\"{u}\c{c}\"{u}kyavuz}
\FirstAuthorAddress{Department of Integrated Systems Engineering,
The Ohio State University, Columbus, OH, 43210, U.S.A. \\{\tt liu.2738@osu.edu,
kucukyavuz.2@osu.edu}}


\begin{document}

	\maketitle 
	
		\begin{abstract}
		 	We study the polyhedral structure of the static probabilistic lot-sizing  problem and propose valid inequalities that integrate information from the chance constraint and the binary setup variables. We prove that the proposed inequalities subsume existing inequalities for this problem, and they are facet-defining under certain conditions. In addition, we show that they give the convex hull description of a related stochastic lot-sizing problem. We propose a new formulation that exploits the simple recourse structure, which significantly reduces the number of variables and constraints of the deterministic equivalent program. This reformulation  can be applied to  general chance-constrained programs with simple recourse. The computational results show that the proposed inequalities and the new formulation are effective for the the static probabilistic lot-sizing  problems.
		\end{abstract}
	\section{Introduction}
	
	In this paper, we study the static probabilistic lot-sizing (SPLS) problem. Given a joint probability  distribution of random demand over a finite planning horizon, and a service level, $1 - \epsilon$,   SPLS problem aims to find a production plan at the beginning of the planning horizon (before the random demand is realized), so that the  expected total cost of production and inventory is minimized, and the probability of stockout does not exceed $\epsilon$. In this study, we focus on finite probability spaces.

	\cite{on2} introduce the   {\it deterministic} uncapacitated lot-sizing (ULS) problem (\textit{without}  backlogging), which is the problem of finding the optimal plan of production and inventory quantities,  to satisfy the demand in each period of the planning horizon on time.  The authors propose an $O(n^2)$ algorithm for ULS, where $n$ is the number of time periods in the planning horizon. Improved polynomial algorithms  can also be found in \cite{logn} and \cite{nlogn}.  \cite{ls} give a complete linear description of the convex hull of ULS in the original space of variables by the so-called $(\ell, S)$ inequalities.     In addition, \cite{extended} propose  an extended formulation for the ULS problem, which gives the complete linear description of the convex hull of solutions to ULS in the extended space.
	
	 \cite{backlogging} provide the first polyhedral study  of a closely related deterministic ULS problem {\it with} backlogging (ULSB),  in which backorders are allowed in intermediate periods and penalized by shortage costs, and demands must be met at the end of the planning horizon. \cite{convexhull} propose a class of inequalities that generalizes the inequalities of \cite{backlogging} and show that this class of inequalities is enough to give a complete linear description of the convex hull of ULSB.  \cite{gk13} give extended formulations for the deterministic ULSB problem when there is a limit on the number of periods in which shortages occur. 
	 
	 The aforementioned studies assume that the demands are known for each time period of the planning horizon. However, in many applications, these parameters are uncertain, and only the joint  probability distribution of these data is  available.  \cite{yongpei} address a multi-stage stochastic integer programming formulation of the uncapacitated lot-sizing problem under uncertainty. They extend the deterministic $(\ell, S)$ inequalities to the stochastic case. \cite{Guan2006} show that these inequalities are sufficient to describe the convex hull of solutions to the  two-period problem  \cite[see, also][]{DiSumma2008}.  \cite{ahmed03} use the tight extended formulation proposed for the deterministic lot-sizing problem to strengthen the deterministic equivalent formulation of the stochastic lot-sizing problem.    \cite{miller}, \cite{huang} and \cite{ruiwei} propose dynamic  programming algorithms  for solving stochastic uncapacitated lot-sizing problems that run in polynomial time in the input size (number of scenarios and time periods). 
	
	The stochastic lot-sizing model assumes that we have to satisfy the uncertain demand in each time period for every scenario, which may lead to an over-conservative solution with excessive inventory.
	As an alternative, for a given service level, $1-\epsilon$, a chance-constrained lot-sizing formulation, referred to as the {\it static} probabilistic lot-sizing problem (SPLS), ensures that the production schedule, which is   determined at the beginning of the planning horizon before the realization of random demands, meets the demands on time with probability at least $1-\epsilon$.   \cite{2002b} consider a variant of SPLS, where    the total expected cost is approximated by eliminating the holding cost and inventory variables from the objective function. The authors 
	propose a branch-and-bound method that relies on a partial enumeration of the so-called $p$-efficient points \cite[see][]{P90,S92,Concavity,R02}.  \cite[See, also][for a more general probabilistic production and distribution planning problem]{LR07}. The SPLS model {\it with} the inventory costs is solved using a  branch-and-cut algorithm  in \cite{kucukyavuz:12}.     
	\cite{Minjiao} propose  a {\it dynamic} probabilistic lot-sizing model, in which the production schedule is updated based on the scenario realization of the previous time periods.  We refer the reader to \cite{simgebook} for a survey on deterministic, stochastic and probabilistic lot-sizing models.

	Chance-constrained programming (CCP) is a class of optimization problems where the probability of an undesirable outcome is limited by a given threshold, $\epsilon$, (see, e.g.,   \citet{charnes:58}, \citet{charnes:63}, \citet{Miller1},  \citet{P73}). \citet{sample} propose sample-average approximation (SAA)  algorithm for CCPs with general probabilistic distribution  \cite[see, also][]{Cal1,Cal2,Nem,Campi}. The resulting sampled problem can be formulated as a large-scale deterministic mixed-integer program.  However, the weakness of the linear programming relaxation of this  formulation  makes it inefficient to solve with commercial integer programming solvers.

	For unstructured chance-constrained programs (CCP) with random right-hand sides,  \citet{luedtke:mp10},  \citet{kucukyavuz:12} and \cite{abdi} study strong valid inequalities for the deterministic equivalent formulation of the chance constraint. In addition, \citet{luedtke:cc13} and \cite{liu} propose decomposition algorithms for two-stage CCPs with a finite number of scenarios, which show significant improvement in computational performance when solving the deterministic equivalent formulation of the CCPs. CCPs with other special structures are also studied in \cite{SL13},  \cite{Song} and \cite{wind}.
	
	In this paper, we provide a polyhedral study of the static probabilistic lot-sizing problem. Different from earlier studies (summarized in Section \ref{sec:prior}), we derive a class of valid inequalities that synthesize  information from the binary production setup variables and the chance constraint (Section \ref{sec:new1}). As a result, we obtain inequalities that are stronger than those considering the chance constraint and lot-sizing structures separately. We prove that our inequalities are facet-defining under certain conditions. Furthermore, we show that they are sufficient to provide the complete linear description of a related stochastic lot-sizing problem. In Section \ref{sec:newformulation}, we propose a new formulation for SPLS, which greatly reduces the number of variables and constraints of the deterministic equivalent formulation. We also show that the proposed new formulation can be extended to general two-stage chance-constrained programs with simple recourse.   Our computational experiments summarized in Section \ref{sec:comp} show that the proposed methods are effective. 
	
	\section{Problem Formulation}

	Given a planning horizon with length $n$, let $N := \{1, \hdots , n\}$. Also, let $x_i$ be the production setup variable and  $f_i$ be the fixed cost of production at  time period  $i$ for all $i \in N$. In addition, let $y_i$ be the production quantity  and  $c_i$ be the unit cost of production at time period $i$, for all $i \in N$. Let $\xi$ be the uncertain demand. Throughout, we   let $[j] = \{1, 2, \dots , j\}$, for $j \in \mathbb Z_+$.

	The generic model of the static probabilistic lot-sizing problem, which is introduced in \cite{2002b}, can be formulated as a two-stage optimization problem. The first-stage problem is stated as:
	\begin{subequations}
		\label{eq:gen}
		\begin{align}
			\min \; &\vf^\top \vx + \vc^\top \vy + \mathbb  E_\xi  \Big(\Theta_\xi( \vy) \Big) \\
			\label{eq:gen1} \text{s.t.}\; & \mathbb P\big(\sum_{i = 1}^t y_i \geq \sum_{i = 1}^t \xi_i, \; t \in N \big) \geq 1 - \epsilon \\
			\label{eq:gen3}&y_i \leq M_i x_i, & i \in N \\
			\label{eq:gen6}&  \vy  \in \mathbb R_+^n, \vx \in \mathbb B^n,
		\end{align}
	\end{subequations}
	where $M_i$ is a large constant to make \eqref{eq:gen3} redundant when $x_i$ equals to one, for all $i \in N$.  Constraint \eqref{eq:gen1} enforces that the probability of violating the demands from time 1 to $n$ should be less than the user-given risk rate $\epsilon$.   In addition, $\Theta_\xi(\vy)$ is the value function of the second-stage problem given by:		
	\begin{subequations}
		\label{eq:gens}
		\begin{align}
			\Theta_\xi( \vy) = \;	\min \; &  \vh^\top \vs(\xi)  \\
			\label{eq:gens4}& s_t(\xi) \geq \sum_{i = 1}^t(y_i - \xi_i)  &  t \in N\\
			\label{eq:gens6}&\vs(\xi)  \in \mathbb R_+^n,
		\end{align}
	\end{subequations}
	where $\vs(\xi)$ is the vector of second-stage inventory  variables with nonnegative cost vector $\vh$.  In addition, constraints \eqref{eq:gens4}  together with \eqref{eq:gens6} ensure the correct calculation of the inventory level.  Note that 	the second-stage problem has a simple-recourse structure. \cite{Minjiao} propose a related model, in which there may be shortages in the intermediate time periods, but  all demand  must be satisfied by the end of the planning horizon to meet contractual  obligations. Our methods are valid for both variations of  SPLS.
	
	Given a finite scenario set $ \Omega = \{1, \hdots , m\}$, let $\pi_j$ be the probability of scenario $j$, for all $j \in \Omega$. In addition, let $d_{ji}$ be the demand for period $i$ under scenario $j$, for all $i \in N$ and $j \in \Omega$. Let $s_{jt}$ be the inventory at the end of time period $t\in N$ in scenario $j\in \Omega$, which incurs a unit holding cost $h_t$. As is common in SAA methods, throughout the rest of the paper, we assume that each scenario is equally likely, i.e., $\pi_j = \frac{1}{m}$, for all $j \in \Omega$.  Letting $k =\lfloor m\epsilon \rfloor$,  the deterministic equivalent formulation of 
	SPLS is 
	\begin{subequations}
		\label{eq:dep}
		\begin{align}
			\min \; &\vf^\top \vx + \vc^\top \vy + \frac{1}{m} \sum_{j = 1}^m   \vh^\top \vs_j \\
			\label{eq:dep1}	\text{s.t.}\; & \sum_{i = 1}^t y_i \geq \sum_{i = 1}^t d_{ji}(1 - z_j), \; & t \in N, j \in \Omega \\
			\label{eq:dep2}	&\sum_{j = 1}^m  z_j \leq k \\
			\label{eq:dep4}		&y_i \leq M_i x_i, & i \in N \\
			\label{eq:dep5}		& s_{jt} \geq \sum_{i = 1}^t(y_i - d_{ji})  &  t \in N, j \in \Omega\\
			\label{eq:dep6}		&\vs_j\in \mathbb R^n_+,  j \in \Omega , \vy  \in \mathbb R_+^n,\vx \in \mathbb B^n,  \vz \in \mathbb B^m,
		\end{align}	
	\end{subequations}
	where we introduce additional logical variable $z_j$, which equals  0 if the demand in each time period under scenario $j$ is satisfied, and 1 otherwise, for all $j \in \Omega$.  In addition, $M_i = \max_{j \in \Omega} D_{jin}$, for all $i \in N$, where $D_{jin} = \sum_{p = i}^n d_{jp}$, for all $j \in \Omega$. Furthermore, the cardinality constraint \eqref{eq:dep2} along with the big-$M$ constraint \eqref{eq:dep1} represents the chance constraint in the equal probability case.  However, this deterministic equivalent formulation is hard to solve due to the  large number of scenario-based variables and constraints, and the big-$M$ type of constraints \eqref{eq:dep1} and  \eqref{eq:dep4}, which yield  weak linear programming relaxations. In the next section, we survey the existing valid inequalities for this class of problems, and then  propose  new valid inequalities.

	\exclude{
		As observed by  \cite{luedtke:cc13}, we can apply mixing inequalities to strengthen the constraints \eqref{eq:dep1}, which have the so-called {\it mixing structure} \cite[]{Gunluk}. The remaining question is how to strengthen  constraint  \eqref{eq:dep4} based on the structure of the lot sizing problem. 
	}

	\section{Valid Inequalities}
	
	In this section, we propose a class of strong valid inequalities for SPLS that subsume known inequalities for this problem.   Before we describe the proposed inequalities, we review existing inequalities for SPLS adapted from the  $(\ell, S)$ inequalities for the deterministic lot-sizing problem, and the mixing inequalities for the deterministic equivalent of chance-constrained programs with random right-hand sides. 	
	
	\subsection{Existing Studies} \label{sec:prior}
	
	Consider the feasible region of \eqref{eq:dep} in the space of $ ({\mathbf x, \mathbf y, \mathbf z})$ variables. Let  $P  =   \{ ({\mathbf x, \mathbf y, \mathbf z}) \in \mathbb B^n \times  \mathbb R_+^n \times \mathbb B^m  \;|\;  \eqref{eq:dep1} - \eqref{eq:dep4} \}.$		
	First, note that we can  adapt the  $(\ell, S)$ inequalities \cite[]{ls} for the deterministic lot-sizing problem, to obtain the following valid inequalities for its chance-constrained counterpart:
	\begin{align}
		\label{eq:weakcut2}
		\sum_{i \in S} y_i + \sum_{i \in \bar S}D_{ji\ell}x_i \geq D_{j1\ell}(1 - z_j), \quad j \in \Omega,
	\end{align}
	where $\ell \in N$, $S \subseteq [\ell]$, and  $\bar S =  [\ell] \setminus S$. To see the validity of \eqref{eq:weakcut2}, note that if $z_j = 0$, then the demand in each time period of the $j$-th scenario must be met, and \eqref{eq:weakcut2} reduces to $(\ell, S)$ inequalities for the $j$-th scenario. Otherwise, if $z_j = 1$, then the inequality is trivially valid.   However, this class of inequalities contains the undesirable big-$M$ terms, which lead to weak  linear programming relaxations. 
	Furthermore, they only contain information from a single scenario at a time. We will address the question on the strength of inequalities \eqref{eq:weakcut2} for a special case in Proposition \ref{prop:k=0}. Similarly, we can also apply the modified extended formulation of deterministic uncapacitated lot-sizing problem studied in \cite{extended} to the SPLS, with the added big-M terms. However, this simple adaption only uses the information from a single scenario, which may not be strong for the deterministic equivalent program  where we have to consider the intersection of the whole scenarios set. In addition, the number of variables and constraints  explode as we increase $m$ and $n$.

	\exclude{of the deterministic equivalent program; second, even assuming that $\epsilon = 0$, i.e., $z_j = 0$, for all $j \in \Omega$, and  the big-$M$ term disappears in \eqref{eq:weakcut2},  the $(\ell, S)$ inequalities for a single scenario are not sufficient to define the convex hull of the deterministic equivalent formulation for a stochastic lot-sizing problem. 	}

	Second, since the big-$M$ inequalities \eqref{eq:dep1} contain the mixing structure, we can apply the mixing inequalities to strengthen the linear programming relaxation of \eqref{eq:dep}.   
	To simplify notation, let $D_{ji} = D_{j1i}$, for all $i \in N$ and $j \in \Omega$.  In addition, for all $i \in N$ and $j \in \Omega$, let $\sigma$ be a permutation of the scenarios such that	
	$		D_{\sigma_{i(1)}i} \geq D_{\sigma_{i(2)}i}  \geq \cdots \geq D_{\sigma_{i(m)}i}, $ 
	where $D_{\sigma_{i(j)}i}$ is the $j$-th largest cumulative demand for the $i$-th time period. To further simplify the notation, let $ D_{\sigma_{i(j)}} =D_{\sigma_{i(j)}i} $. Let  $T^*_i = \{\sigma_{i(1)}, \sigma_{i(2)}, \hdots , \sigma_{i(k)}\}$, for all $i \in N$. Throughout the paper, when we define a set such as $T:=\{t_1,t_2,\ldots,t_a\}$, it should be understood that $a$ is the cardinality of $T$. 
	\begin{proposition}[adapted from \cite{luedtke:cc13}]
		For $\ell \in N$,  let $T_\ell := \{t_{\ell(1)}, t_{\ell(2)}, \hdots , t_{\ell(a_\ell)}\} \subseteq T_\ell^*$,  where 
		$D_{t_{\ell(1)}} \geq D_{t_{\ell(2)}} \geq \cdots \geq D_{t_{\ell(a_\ell)}} $. The {\it basic} mixing inequalities
		\begin{align}\label{eq:mixing}
			\sum_{i = 1}^\ell y_i + \sum_{j = 1}^{a_\ell} (D_{t_{\ell (j)}} - D_{t_{\ell (j + 1)}})  z_{t_{\ell (j)}}\geq  D_{t_{\ell (1)}},				
		\end{align}
		are valid for $P$, where $t_{\ell (a_\ell + 1)} = \sigma_{\ell (k + 1)}$.
	\end{proposition}
	\cite{luedtke:cc13}, \cite{kucukyavuz:12} and \cite{abdi} provide extensions of the basic mixing inequalities \eqref{eq:mixing} for equal and general probability cases. However, the mixing  inequalities based on cumulative production quantities do not provide any strengthening for fractional $\vx$. 	 Hence, an interesting research question is whether we can combine the mixing inequalities and the $(\ell, S)$ inequalities to obtain valid inequalities that cut off fractional ($\vx,\vz$). Next, we provide an affirmative  answer to this question. 
	
	\subsection{New Valid Inequalities} \label{sec:new1}
	In this section, we propose a class of valid inequalities which subsumes inequality  \eqref{eq:mixing}. In addition, we study the strength of the new inequalities and provide a polynomial separation algorithm, which is exact under certain conditions.
	

	\begin{proposition}
		For $\ell \in N$, let $S \subseteq [\ell]$, $\bar S = [\ell] \setminus S$, and let $T_{i-1} := \{t_{i-1(1)}, t_{i-1(2)}, \hdots , t_{i-1(a_{i-1})}\} \subseteq T_{i-1}^*$,  where  
		$D_{t_{i-1(1)}} \geq D_{t_{i-1(2)}} \geq \cdots \geq D_{t_{i-1(a_{i-1})}}$, for all $ i \in (\bar S\setminus\{1\})\cup\{\ell+1\}$. In addition, we fix $t_{\ell(1)} = \sigma_{\ell(1)} $.   Let $\bar T=(\cup_{i \in \bar S} T_{i - 1} )\cup  T_\ell$.  The inequality
		\begin{align}
			\label{eq:newvalid}
			\sum_{i \in S} y_i  + \sum_{i \in \bar S}(D_{t_{\ell(1)}} - D_{t_{i - 1(1)}}) x_i + \sum_{j \in \bar T} \bar \alpha_j z_j \geq  D_{t_{\ell(1)}}
		\end{align}
		is valid for $P$, 	where ${t_{i-1(a_{i-1} + 1)}} = {\sigma_{i-1(k + 1)}}$, for all $i \in (\bar S\setminus\{1\})\cup\{\ell+1\}$, 
		\begin{align*}
			\alpha_{ji}  = 
			\begin{cases}
				0,   &\text{ if } i=\ell, j\not\in  T_\ell, \text{ or if } i+1\in \bar S, j\not\in  T_{i}, \\
				D_{t_{\ell(p)}} - D_{t_{\ell(p + 1)}}, & \text{ if } i=\ell, j= t_{\ell(p)} \in  T_\ell \text{ for some } p \in[a_\ell], \\
				D_{t_{i(p)}} - D_{t_{i (p + 1)}}, & \text{ if }  i+1 \in \bar S, j = t_{i (p)} \in T_{i } \text{ for some } p \in[a_{i}], 
			\end{cases}
		\end{align*}
		and $\bar \alpha_{j}  =   \max \Big\{	\max_{i \in \bar S} \{\alpha_{j(i-1)} \},  \alpha_{j\ell}  \Big\} $ for $j\in \bar T$.

		\exclude{
			\begin{align*}
				\alpha_{j}  = 
				\begin{cases}
					0,     \quad \;\; \text{if } j \not \in (\cup_{i \in \bar S} T_{i - 1} )\cup  T_\ell, \\
					\max \Big\{	\max_{i \in \bar S} \{\alpha_{ji-1} \},  \alpha_{j\ell}  \Big\}  , \quad \;\; \text{otherwise, }
				\end{cases}
			\end{align*}
			\comment{if a scenario $j$ does not belong to $T_\ell$ or $T_{i - 1}$, for all $i \in \bar S$, then it does not show up in the left-hand side in our inequality. Otherwise, the coefficient of scenario $j$ equals to the maximum of the ``mixing coefficient'' of scenario $j$ obtained from $T_\ell$ and $T_{i - 1}$, for all $i \in \bar S$}	where $\alpha_{j\ell} = 0$ if $j \not \in T_\ell$. Otherwise, $\alpha_{j\ell} = D_{t_{\ell(p)}} - D_{t_{\ell(p + 1)}}$,   where $j = t_{\ell(p)} \in T_\ell$, for some  $p = 1, 2, \dots, a_\ell $. Similarly, for all $i \in \bar S$, let $\alpha_{ji - 1} = 0$, if $j \not \in T_{i - 1}$.  Otherwise, let $\alpha_{ji - 1} = D_{t_{i - 1(p)}} - D_{t_{i - 1(p + 1)}}$,  where $j = t_{i - 1(p)} \in T_{i - 1}$, for some  $p = 1, 2, \dots, a_{i - 1} $.
		}
	\end{proposition}
	
	\begin{proof}
		Suppose that $x_i = 0$, for all $i \in \bar S$. Then inequality \eqref{eq:newvalid} reduces to:
		\begin{align*}
			\sum_{i \in S}  y_i   + \sum_{j \in \bar T} \bar \alpha_j z_j     
			& \geq  \sum_{i \in S}  y_i   + \sum_{j = 1}^{a_\ell} (D_{t_{\ell (j)}} - D_{t_{\ell (j + 1)}})  z_{t_{\ell (j)}} 	\geq   D_{t_{\ell(1)}},
		\end{align*}	
		where the first inequality follows from the definition of $\bar \alpha_j$, and the second inequality follows from the validity of the mixing inequality \eqref{eq:mixing} for  time period $\ell$ when $x_i=0$ for all $i\in \bar S$.  	Otherwise, let $i' \in \bar S$ be the smallest index in $\bar S$ such that $x_{i'} = 1$. Then we have:
		\begin{align*}
			\sum_{i \in S} y_i  &+ \sum_{i \in \bar S \setminus \{i'\}} (D_{t_{\ell(1)}} - D_{t_{i - 1(1)}}) x_i + \sum_{j\in \bar T}  \bar \alpha_j z_j   
			&\geq \sum_{i \in S} y_i +   \sum_{j = 1}^{a_{i' - 1}} (D_{t_{i' - 1 (j)}} - D_{t_{i' - 1(j + 1)}} ) z_{t_{i' - 1(j)}}   \geq D_{t_{i' - 1(1)}},
		\end{align*}
		where the first inequality follows from the nonnegativity of $(D_{t_{\ell(1)}} - D_{t_{i - 1(1)}})$ and the definition of $\bar \alpha_j$, since $t_{\ell(1)} = \sigma_{\ell(1)} $.  In addition, the second inequality follows  from the validity of mixing inequality \eqref{eq:mixing} for period $i'$ given that $y_{i} = 0$ for $i\in \bar S, i<i'$. 
	\end{proof}

	\begin{example}\label{ex:ex1}
		
		Let $k = 2$, $n = 2$,  $m = 5$,  and consider the demand data given in Table \ref{tab:1}.		
					\begin{table}[hbt]
						\centering
						\caption{Data for Example \ref{ex:ex1}.}
						\setlength{\tabcolsep}{10pt}
						\label{tab:1}	\begin{tabular}{l*{6}{c}r}
							$ \text{Scenarios} $              & 1 & 2 & 3 & 4 &5 \\
							\hline
							$d_1$ & 6 & 3 & 1 & 2 &4 \\
							$d_2$            & 1 & 6 & 10 & 8   &5\\
							$d_1 + d_2$          & 7 & 9 & 11 & 10 &  9 \\
						\end{tabular} 	
					\end{table}

		 For $\ell = 2$,  let $T_\ell = \{3, 4 \}$, $T_{\ell - 1} =  T_{1}= \{1, 5\}$, $S= \{1\}$, and $\bar S = \{2\}$.	 According to the definition,  $\bar T = \{1, 3, 4, 5\}$. 
			Since $1 \in \bar T$, and scenario 1 is the scenario with the largest demand in the first time period, $\alpha_{11} = D_{t_{1(1)}} - D_{t_{1(2)}} = D_{\sigma_{1(1)}} - D_{\sigma_{1(2)}} = 6 - 4 = 2$. In addition, since $1 \not \in T_\ell$, we have $\alpha_{12} = 0$. Hence, we have $\bar \alpha_1 = \max \{\alpha_{11}, \alpha_{12}\} = 2$.  Since $3 \in \bar T$, and $3 \not \in T_1$, we have $\alpha_{31} = 0$. In addition, since $3 \in T_\ell$, and it is the scenario with the largest cumulative demand at time period $\ell$, we have $\alpha_{32}= D_{t_{2(1)}} - D_{t_{2(2)}}  = D_{\sigma_{2(1)}} - D_{\sigma_{2(2)}} = 11 - 10 = 1$. Hence, we have $\bar \alpha_{3} =\max \{\alpha_{31}, \alpha_{32}  \} = 1 $. Similarly, $\bar \alpha_4 = \max \{\alpha_{41}, \alpha_{42}  \} = \max \{0, D_{t_{2(2)}} - D_{t_{2(3)}}   \}   = \max \{0, D_{\sigma_{2(2)}} - D_{\sigma_{2(3)}}   \} = 1$, and   $\bar \alpha_5 = \max \{\alpha_{51}, \alpha_{52}  \} = \max \{D_{t_{1(2)}} - D_{t_{1(3)}},  0  \}   = \max \{D_{\sigma_{1(2)}} - D_{\sigma_{1(3)}},  0  \} = 1  $. Hence, the proposed inequality for this choice of parameters is:
		\begin{align}
			\label{eq:ex3.1}
			y_1 +  5x_2  + 2z_1 + z_3 + z_4 +  z_5 \geq 11.
		\end{align}

	\end{example}
	
	Next  we show the strength of the proposed inequalities. 
	\begin{proposition}	\label{prop:eq1facet}
		Inequalities \eqref{eq:newvalid} are facet-defining for $conv(P)$ if		
		\begin{enumerate}
			\item   $\bar S \not = \emptyset$ and $1 \in S$;
			\item  $T_{p - 1} \cap (T^*_{q - 1} \cup T^*_\ell  )= \emptyset$,  $T_{\ell} \cap T^*_{q - 1} = \emptyset$, for all $p \not = q$, and $p, q \in \bar S $;
			\item  $D_{t_{\ell(1)}} - D_{t_{i - 1(a_{i - 1} + 1)}} < M_i$, for all $i \in \bar S$;
			\item $d_{ji} > 0$, for all $j \in \Omega$ and $i \in N$. 			
		\end{enumerate}	
	\end{proposition}
	\begin{proof}
		First, note that $dim(P) = 2n + m - 1$, assuming that $d_{ji} > 0$, for all $j \in \Omega$ and $i \in N$, since $x_1 = 1$ when demands are positive at period 1, and  backordering is not allowed in $n - k$ scenarios.	To show that  inequality \eqref{eq:newvalid} is facet-defining under conditions {\it (i)-(iv)}, we need to find $2n + m - 1$ affinely independent points $(\mathbf x, \mathbf y, \mathbf z)$ that satisfy \eqref{eq:newvalid} at equality.  
		Let $g(t_{i(j)})$, for all $t_{i(j)} \in T_i$ and $i+1 \in \bar S\cup\{\ell+1\}$,  be a unique mapping such that scenario $t_{i(j)}$ has the $g(t_{i(j)})$-th largest cumulative demand at time period $i$.	 Also,   for $p \in [\ell]$, $i\in \bar T$ and $j \in [a_i + 1]$, let  $\mathbf{\bar y}_{j}^p$ be an $n$-dimensional vector such that $\bar y^p_{j1} = D_{t_{p(j)}}$  and  $\bar y^p_{ji} = 0$, for all $i = 2, \hdots , n$. 
		
		First, consider the feasible points: 
		$ (\mathbf e_1 + \ve_{\ell+ 1} , \mathbf{\bar y}_{j}^\ell + \ve_{\ell+ 1} M_{\ell + 1} , \sum_{i = 1}^{g(t_{\ell(j)}) - 1}  \mathbf {e}_{\sigma_{\ell(i)}}),$ for $ j  \in [a_\ell + 1]$,
		where $\mathbf e_j$ is the $j$-th unit vector with appropriate dimension. These $a_\ell + 1$ points 
		are affinely independent and satisfy inequality \eqref{eq:newvalid} at equality. 
		Next,  consider the  set of points: 
		$	 (\mathbf e_1 + \ve_{\ell + 1}, \mathbf{\bar y}_{1}^\ell + \ve_{\ell+ 1} M_{\ell + 1}, \mathbf {e}_j ),$ $ \forall j = \Omega \setminus \bar T.$
		These $m - a_\ell - \sum_{i \in \bar S} a_{i - 1}$ points are feasible, affinely independent from all other points, 
		and satisfy inequality \eqref{eq:newvalid} at equality.
		
		Next, for $T_{i - 1}$, for all $i \in \bar S$ we construct the following set of feasible points: 
		\exclude{	
			Next, let $\sigma_{i(p^*_i)}$ be the index of $t_{i(1)}$ in the set $T_i^*$, i.e., $t_{i(1)}$-th scenario has the $p^*_i$-th largest cumulative demand in period $i$. Then,  based on each $T_i$, we construct the following set of feasible points (\textbf{PTT$_i$}), for all $i \in \bar S$. :
		}		
		\begin{align*}
			\nonumber		& (\mathbf e_1 + \mathbf e_i + \ve_{\ell + 1}, \mathbf{\bar y}^{i - 1}_{j} + \mathbf e_i (D_{\sigma_{\ell(1)}} - D_{t_{i - 1(j)}}) + \ve_{\ell + 1} M_{\ell + 1}, \sum_{p = 1}^{g(t_{i - 1(j)})-1} \mathbf e_{\sigma_{i - 1(p)}}),   \quad  j \in [a_{i - 1} + 1],  
			\end{align*}
		 and
					\begin{align*}	
			\nonumber		&(\mathbf e_1 + \mathbf e_i +  \ve_{\ell + 1}, \mathbf{\bar y}^{i - 1}_{a_{i - 1} + 1} + \mathbf e_i (D_{\sigma_{\ell(1)}} - D_{t_{i - 1(a_{i - 1} + 1)}} + \triangle) + \ve_{\ell + 1} M_{\ell + 1} , \sum_{p = 1}^{k} \mathbf {e}_{\sigma_{i - 1(p)}} ),
		\end{align*}
		where  $0<\triangle \leq M_i-D_{\sigma_{\ell(1)}} + D_{t_{i - 1(a_{i - 1} + 1)}} $.	These $\sum_{i \in \bar S} a_{i - 1} + 2 |\bar S|$ points 
		are affinely independent from all other points, and satisfy inequality \eqref{eq:newvalid} at equality. 
		
		Next, for each $i \in S  \setminus \{1\}$, we construct the following set of feasible points:	
		\begin{align*}
			\nonumber	&(\mathbf e_1 + \mathbf e_i + \ve_{\ell + 1}, \mathbf{\bar y}_{1}^\ell  + \ve_{\ell + 1} M_{\ell + 1}, \mathbf 0), \\
			&(\mathbf e_1 + \mathbf e_i + \ve_{\ell + 1}, \ve_1 D_{\sigma_{i - 1}(1)} + \mathbf e_i (D_{\sigma_{\ell(1)}} - D_{\sigma_{i - 1}(1)}) + \ve_{\ell + 1} M_{\ell + 1}, \mathbf 0).
		\end{align*}	
		These $2|S| - 2$ points are feasible, affinely independent from all other points, and satisfy inequality \eqref{eq:newvalid} at equality. 
		\exclude{	
			Finally, we obtain the last $2|S| + 2(n - |L|)- 2$ feasible affinely independent points as follows:
			\begin{align*}
				\nonumber	&(\mathbf e_1 + \mathbf e_i, \mathbf{\bar y}_{1}^\ell + \mathbf e_i (D_{\sigma_{\ell(1)}} - D_{t_{i(1)}}), \mathbf 0), \quad i \in( N \setminus L) \cup (S  \setminus \{1\})\\
				\label{eq:add}		&(\mathbf e_1 + \mathbf e_i, \mathbf{\bar y}_{1}^\ell + \mathbf e_i (D_{\sigma_{\ell(1)}} - D_{t_{i(1)}} + \epsilon), \mathbf 0), \quad i \in( N \setminus L) \cup (S  \setminus \{1\})
			\end{align*}
			where $\mathbf 0$ is an $n$-dimensional zero vector. These $2|S| + 2(n - |L|)- 2$ points satisfy inequality \eqref{eq:newvalid} at equality, and are affinely independent from all other points listed. Hence, we obtain $2n + m - 1$ linearly independent feasible points that satisfy \eqref{eq:newvalid} with equality, which completes the proof.}
		
		Next, for each $i \in N \setminus[\ell+1]$, consider the following set of points:     
		$	 (\mathbf e_1 + \ve_{\ell+ 1}+ \ve_i, \mathbf{\bar y}_{1}^\ell + \ve_{\ell+ 1} M_{\ell + 1},  \mathbf 0), $ and 
		$ (\mathbf e_1 + \ve_{\ell+ 1}+ \ve_i, \mathbf{\bar y}_{1}^\ell + \ve_{\ell+ 1} M_{\ell + 1}+ \ve_i \triangle,  \mathbf 0),$
		\exclude{
			\begin{align*}
				\nonumber	&(\mathbf e_1 + \mathbf e_{p^*} + \ve_i,    \ve_1 D_{\sigma_{\ell(1)}} + \ve_{p^*} M_{p^*}, \mathbf 0), \quad i \in N \setminus [\ell], \\
				&(\mathbf e_1 + \mathbf e_{p^*} + \ve_i,    \ve_1 D_{\sigma_{\ell(1)}} + \ve_{p^*} M_{p^*} + \ve_i \triangle, \mathbf 0), \quad i \in N \setminus  [\ell] ,
			\end{align*}
		}
		where $0<\triangle \leq M_i$.	These $2(n - \ell-1)$ points are feasible, affinely independent from all other points, and satisfy inequality \eqref{eq:newvalid} at equality.  
		
		Finally, for  a fixed index $p^* \in \bar S$,  we construct the remaining two points:
		$		(\mathbf e_1 + \ve_{p^*}, \mathbf{\bar y}_{1}^{p^* - 1} + \ve_{p^*} M_{p^*},  \mathbf 0), $ and 
		$  (\mathbf e_1 + \ve_{p^*} + \ve_{\ell + 1}, \mathbf{\bar y}_{1}^{p^* - 1} + \ve_{p^*} M_{p^*} + \ve_{\ell + 1}\triangle,  \mathbf 0),$ 
		where $ 0< \triangle < M_{\ell + 1}$. These two points are feasible, affinely independent from all other points, and satisfy inequality \eqref{eq:newvalid} at equality. 
		Hence, we obtain $2n + m - 1$ affinely independent feasible points that satisfy  inequality \eqref{eq:newvalid} at equality, which completes the proof.	 
	\end{proof}
	\exclude{
		\begin{corollary}
			Suppose that 	
			\begin{itemize}
				\item  If $t_{\ell(1)} = \sigma_{l(1)}$ , and $1 \in S$;
				\item  Let $S' \subset \bar S$, if $T_i = T_l$, for all $i \in S'$, and $T_p \cap T^*_q = \emptyset$, for all $p \not = q$, and $p, q \in (\bar S \setminus S') \cup\{\ell\}$;
				\item If $D^{t^\ell_{1}}_\ell - D^{t^i_{ k + 1}}_i < M_i$.
			\end{itemize}
			Then the proposed inequalities are facet-defining.
			
		\end{corollary}
	}
	
	\renewcommand{\theexample}{\ref{ex:ex1}}
	\begin{example} (Continued.) 
		 Inequality \eqref{eq:ex3.1} is a facet-defining inequality for $conv(P)$, because $T_{\ell - 1} \cap T^*_{\ell} = \emptyset$, $T_{\ell } \cap T^*_{\ell - 1} = \emptyset$, $1 \in S$, and $D_{t_{2(1)}} - D_{t_{1(1)}} = D_{\sigma_{2(1)}} - D_{\sigma_{1(1)}}  = 5 < M_2 = 10$.
		
	\end{example}

	\begin{remark}
		Note  that if $\bar S  = \emptyset$, then the proposed inequality \eqref{eq:newvalid} reduces to the mixing inequality \eqref{eq:mixing} for a given $\ell \in N$ and $T_\ell\subseteq T^*_\ell$.
		In addition,  suppose  that  $D_{{\sigma_{\ell + 1 (k + 1)}}} \geq D_{{\sigma_{\ell (1)}}}$.   Consider   inequality \eqref{eq:newvalid} for the $(\ell + 1)$-th time period, when $\bar S =  \{\ell + 1\} $ and  $T_{\ell + 1} = \emptyset$, for the same choice of  $T_\ell$ as inequality \eqref{eq:mixing}:
		\begin{align} \label{eq:rem-mixlS}
			&\sum_{i = 1}^\ell y_i  + \sum_{j = 1}^{a_\ell} (D_{t_{\ell (j)}} - D_{t_{\ell (j + 1)}})  z_{t_{\ell (j)}}   
			\geq   D_{{\sigma_{\ell + 1 (k + 1)}}} - (D_{{\sigma_{\ell + 1 (k + 1)}}} -  D_{{\sigma_{\ell (1)}}}) x_{\ell + 1}.
		\end{align}
		Because $D_{{\sigma_{\ell + 1 (k + 1)}}} \geq D_{{\sigma_{\ell (1)}}}$ by assumption, the right-hand side of \eqref{eq:rem-mixlS} equals $D_{{\sigma_{\ell + 1 (k + 1)}}}(1-x_{\ell + 1})+D_{{\sigma_{\ell (1)}}} x_{\ell + 1} \geq  D_{\sigma_{\ell (1)}} = D_{t_{\ell (1)}}$, the right-hand side of \eqref{eq:mixing}. 
		\exclude{
			Since $x_{\ell + 1} \leq 1$ and $D_{{\sigma_{\ell + 1 (k + 1)}}} \geq D_{{\sigma_{\ell (1)}}}$, we have: 
			$$ D_{t_{\ell ( 1)}} +   (D_{{\sigma_{\ell (1)}}} - D_{{\sigma_{\ell + 1 (k + 1)}}} ) (x_{\ell + 1} - 1) \geq  D_{t_{\ell (1)}} ,$$
			which indicates that  inequality \eqref{eq:rem-mixlS} is at least as strong as the mixing inequality \eqref{eq:mixing}.  
		}
		Hence, if $\bar S = \emptyset$ and  $D_{{\sigma_{\ell + 1 (k + 1)}}} \geq D_{{\sigma_{\ell (1)}}}$, then the mixing inequality \eqref{eq:mixing} is  dominated by the proposed inequality \eqref{eq:rem-mixlS}.
	\end{remark}
	Next, we consider another special case that shows the strength of our inequalities.
	\begin{proposition} \label{prop:k=0}
		If $\epsilon = 0$, then adding the proposed inequalities \eqref{eq:newvalid} to $P$ is sufficient to give the complete  linear description of  $conv(P)$.
	\end{proposition}
	
	\begin{proof}
		If $\epsilon = 0$, then $k = 0$, and we have to satisfy every scenario, i.e.,  the cumulative production until time period $i\in N$ must be sufficient  to satisfy the scenario with largest cumulative demand until time period $i$. In this case,  $\bar T = \emptyset$ , and the proposed inequalities \eqref{eq:newvalid} reduce to the following inequalities:
		\begin{align}
			\label{eq:reduce}
			\sum_{i \in S} y_i  + \sum_{i \in \bar S}(D_{\sigma_{\ell(1)}} - D_{\sigma_{i-1(1)}}) x_i \geq  D_{\sigma_{\ell(1)}}.
			\end{align}	
		Furthermore, when $k = 0$,  we can fix $\vz = \mathbf 0$ and $s_{jt} = \sum_{i = 1}^t(y_i - d_{ji})$ for all $t\in N, j\in \Omega$, and rewrite the deterministic equivalent program:
		\begin{subequations}
			\label{eq:ssls}	\begin{align}
				\min \; &\vf^\top \vx + \vc^\top \vy + \sum_{j = 1}^m\sum_{t=1}^n \pi_j  h_t (\sum_{i = 1}^t(y_i - d_{ji}) )\\
				\text{s.t.}\; & \sum_{i = 1}^t y_i \geq  D_{\sigma_{t(1)}} &  t \in N\\
				& y_i \leq M_i x_i, & i \in N \\
				& \vy  \in \mathbb R_+^n, \vx \in \mathbb B^n. 
			\end{align}
		\end{subequations}
		Note that the optimization problem \eqref{eq:ssls} is equivalent to a deterministic uncapacitated lot-sizing problem, where the cumulative demand in each time period is given by the largest cumulative demand in each time period over all scenarios. 	Hence,  the $(\ell, S)$ inequalities for the deterministic equivalent program \eqref{eq:ssls} when $k = 0$  are sufficient to describe  $conv(P)$ when $\epsilon = 0$ \cite[follows from][]{ls}, and they are in the form of inequality \eqref{eq:reduce}, which is a special case of the proposed inequality \eqref{eq:newvalid} when $k = 0$. 
	\end{proof}
In contrast, for  the special case of $\epsilon = 0$, when we let $\vz = \mathbf 0$, inequalities
	\eqref{eq:weakcut2} reduce to $(\ell, S)$ inequalities for each scenario  $j\in \Omega$ individually, which is not sufficient to describe $conv(P)$ in this case. Clearly, inequalities  \eqref{eq:newvalid} combine information across all scenarios and yield stronger inequalities.  
	
	\vspace{0.5cm}
	
	\noindent \textbf{Separation of inequalities \eqref{eq:newvalid}}: 
	There are exponentially many  inequalities \eqref{eq:newvalid}. We have two main questions when dealing with the separation problem for a given  $\ell\in N$: first, for any time period $i \not = 1$, we need to decide if $i \in S$ or $i \in \bar S$; second, for each $i \in \bar S$,  we need to find a subset $T_{i-1}$ of $T_{i-1}^*$ so that the term $\sum_{j \in \bar T}\bar \alpha_j z_j$ is minimized. 
	First, given a fractional solution $(\hat \vx, \hat \vy, \hat \vz)$, for all $i \in N\setminus\{1\}$, we solve the following  problems
	\begin{align} 	\label{eq:sep1}
		Y_{i - 1} = \min_{T_{i - 1} \subseteq T_{i - 1}^*}  \;\{ & - D_{t_{i - 1(1)}}\hat x_i +   \sum_{p = 1}^{a_{i - 1}} (D_{t_{i - 1(p)}} - D_{t_{i - 1(p + 1)}})\hat z_{t_{i - 1(p)}} \},
	\end{align}
	\begin{align} 	\label{eq:sep2}
		\hat Y_{i} = \min_{T_{i} \subseteq T_{i}^*, \sigma_{i(1)}\in T_i  }  \;\{  \sum_{p = 1}^{a_{i}} (D_{t_{i(p)}} - D_{t_{i(p + 1)}})\hat z_{t_{i(p)}} \}.
	\end{align}
	Problems \eqref{eq:sep1} and \eqref{eq:sep2} can be solved similarly to the  separation of the mixing inequalities in $O(k\log k)$ time \cite[]{Gunluk} for each $i \in N\setminus\{1\}$.  We let $\bar T_{i- 1}$  and $\hat T_i$ be the optimal argument of problems \eqref{eq:sep1} and \eqref{eq:sep2}, respectively.   
	Finally, for each $\ell \in N\setminus\{1\}$ and $i\in [\ell]$ if $\hat y_i \leq D_{t_{\ell(1)}} \hat x_i+Y_{i-1}$, then we let $i \in S$. Otherwise,  we let  $i \in \bar S$ and $T_{i - 1} = \bar T_{i - 1}$. Then we obtain $\alpha_{j(i - 1)}$ for each $i\in \bar S\cup\{\ell+1\}$ and  $j\in \bar T_{i - 1}$. In addition,  we let $T_\ell=\hat T_\ell$ and  $\bar \alpha_{j} =  \max \Big\{	\max_{i \in \bar S} \{\alpha_{j(i-1)} \},  \alpha_{j\ell}  \Big\}$, for all $j \in \bar T=(\cup_{i \in \bar S} T_{i - 1} )\cup   T_\ell$.  If $\sum_{i\in S} \hat y_i +\sum_{i\in \bar S}( D_{t_{\ell(1)}}- D_{t_{i - 1(1)}}) \hat x_i+\sum_{j\in \bar T}\bar \alpha_j \hat z_j<D_{t_{\ell(1)}}$ for this choice of $\ell, S,\bar T$, then we have found a violated inequality \eqref{eq:newvalid}.

	\begin{proposition}\label{prop:separ}
		The proposed separation procedure runs in $O(n\max\{n,k\log(k)\})$ time.  	Suppose that $T^*_{p - 1} \cap T^*_{q - 1} \cap T_\ell = \emptyset$, for all $p \not = q$, and $p, q \in \bar S $, then the proposed separation procedure is exact.
	\end{proposition}
	\begin{proof}
		For a fixed index $\ell\in N$, if the condition stated in the proposition holds, then we can rewrite inequality \eqref{eq:newvalid} as:
		\begin{align*}
			\sum_{i \in S} y_i  + \sum_{i \in \bar S} \Big ((D_{t_{\ell(1)}} - D_{t_{i - 1(1)}}) x_i + &\sum_{p = 1}^{a_{i - 1}} (D_{t_{i - 1(p)}} - D_{t_{i - 1(p + 1)}})z_{t_{i - 1(p)}}\Big) & +  \sum_{p = 1}^{a_\ell} (D_{t_{\ell(p)}} - D_{t_{\ell(p + 1)}})z_{t_{\ell(p)}} \geq  D_{t_{\ell(1)}},
		\end{align*}
		because $\alpha_{j(i - 1)}=0$, for all but at most one $i \in \bar S\cap\{\ell+1\}$ and $j \in T_{i-1}$. As a result, each time period is separable from other time periods, and the separation procedure is exact. 	
		\exclude{
			If $\hat y_i \leq \min_{T_{i - 1} \subseteq T_{i - 1}^* } \{ (D_{t_{\ell(1)}} - D_{t_{i - 1(1)}}) \hat x_i + \sum_{j = 1}^{a_{i - 1}} (D_{t_{i - 1(p)}} - D_{t_{i - 1(p + 1)}})\hat z_{t_{i - 1(p)}} \} $, then by letting $i \in \bar S$, we obtain the minimum left-hand side of inequality \eqref{eq:newvalid}.  Otherwise, according to the definition of $Y_\ell$ and  $Y_{i - 1}$, for all $i \in \bar S$,  the separation algorithm of the mixing cuts proposed in \cite{Gunluk}, the left-hand side of  inequality \eqref{eq:newvalid} is  minimized by the choice of $\bar T_i$ and $\bar T_\ell$.
		}
		
		The complexity of the algorithm for solving \eqref{eq:sep1} and \eqref{eq:sep2} for all $i \in N\setminus\{1\}$ is $O\big(nk\log(k)\big)$. After finding the optimal $\bar T_{i-1}$ for $i \in N\setminus\{1\}$, which is independent of the choice of $\ell$, identifying the set $S$ for a given $\ell\in N$ takes $O(n)$ time.  Therefore,  we get an overall run time of $O(n\max\{n,k\log(k)\})$.
	\end{proof}
	
	If the conditions in Proposition \ref{prop:separ} are not satisfied, then the separation procedure is a heuristic. 
	
 In Appendix \ref{sec:stock} we give a second class of valid inequalities that involves the inventory variables, which is valid for the deterministic equivalent formulation, and facet-defining under certain conditions.

	\section{A new formulation that exploits the simple recourse property} \label{sec:newformulation}
	
	The deterministic equivalent formulation contains $O(mn)$ additional variables, which becomes computationally challenging if the number of scenarios, $m$,  or the number of time periods, $n$ increases. One can  consider a Benders decomposition algorithm given in Appendix \ref{sec:Benders}. However, we may have to add exponentially many  optimality cuts, which significantly slow down the solution of the master problem, as we show in our computational study in Section \ref{sec:comp}.
	
	In this section, we propose a new formulation for SPLS that is similar to the master problem used in the Benders decomposition algorithm. However  we show that the new formulation only uses polynomially many  inequalities to capture the second-stage cost. 
	
	For all $i \in N$ and $j \in \Omega$, let $\bar \sigma$ be a permutation of the scenarios such that	
	$		D_{\bar \sigma_{i(1)}i} \leq D_{\bar \sigma_{i(2)}i}  \leq \cdots \leq D_{\bar \sigma_{i(m)}i}, $ 
	where $D_{\bar \sigma_{i(j)}i}$ is the $j$-th \textit{smallest} cumulative demand for the $i$-th time period. To further simplify the notation, let $ D_{\bar \sigma_{i(j)}} =D_{\bar \sigma_{i(j)}i} $.

	\begin{proposition}\label{pro:8}
		Let $\Theta'_i$ be an additional variable that captures the total inventory of $i$-th time period for \textit{all} scenarios. In addition, let $[k]^+ = \{0, 1, 2, \dots, k\}$. The  formulation
		\begin{subequations}
			\label{eq:newformulation}	
			\begin{align}
				\min \; &\vf^\top \vx + \vc^\top \vy + \frac{1}{m}\sum_{i = 1}^n h_i \Theta'_i   \\
		\label{eq:15a}		\text{s.t.}\; & \eqref{eq:dep1} - \eqref{eq:dep4},  \\
				\label{eq:strong}		& \Theta'_i \geq (m - q)\sum_{p = 1}^i y_p - \sum_{p = 1}^{m - q} D_{\bar \sigma_{i(p)}}, &  i \in N,   q \in [k]^+ \\
		\label{eq:15c}		& x \in \mathbb B^n, y \in \mathbb R^n_+, z \in \mathbb B^m, \Theta' \in \mathbb R^n_+,
			\end{align}	
		\end{subequations}
		is equivalent to the deterministic equivalent of SPLS \eqref{eq:dep} under equiprobable scenarios.	
	\end{proposition}
	
	\begin{proof}		
	We can rewrite the deterministic equivalent formulation \eqref{eq:dep} as a  two-stage problem  given by	
		\begin{align*}
			\min \; &\vf^\top \vx + \vc^\top \vy + \frac{1}{m}\sum_{i = 1}^n h_i \Theta'_i(\vy)   \\
					\text{s.t.}\; & \eqref{eq:dep1} - \eqref{eq:dep4},  \\
		& x \in \mathbb B^n, y \in \mathbb R^n_+, z \in \mathbb B^m, \Theta' \in \mathbb R^n_+,
		\end{align*}	
	where $\Theta'_i(\vy)$,  the total inventory level at each period $i$,   is defined by the second-stage simple resource  problem with respect to each time period $i$,  stated as 
	\begin{align*}
		\Theta'_i(\vy)  = \; \min & \;  \sum_{j = 1}^m s_{ji} \\
	\text{s.t.}\;  	& s_{ji} \geq \sum_{p = 1}^i y_p - D_{ji}, &  j \in \Omega   \\
	& s_{ji} \geq 0 &  j \in \Omega.
		\end{align*}
	Let $\Theta_i'$ be a variable that captures the correct value of $\Theta'_i(\vy)$ for any feasible $\vy$ 	through the exponentially many inequalities	
		\begin{align}
			\label{eq:weak}   \Theta'_i \geq (m - q)\sum_{p = 1}^i y_p  - \sum_{ j \in  R_q} D_{ji}, \quad  q \in [k]^+,
		\end{align}
		where $R_q \subseteq \Omega $ is a subset of scenarios such that $|R_q| = m - q$. Hence,  to show that the proposed formulation \eqref{eq:newformulation} is equivalent to the deterministic equivalent program \eqref{eq:dep},   we show that the polynomial subclass \eqref{eq:strong} of the exponential class of inequalities \eqref{eq:weak} suffice to give a correct formulation.
			For a fixed $q \in [k]^+$ and $i \in N$,  consider the following chain of inequalities:
		\begin{align*}
		 \Theta_i' \geq (m - q)\sum_{p = 1}^i y_i  - \sum_{p = 1}^{m - q} D_{\bar \sigma_{i(p)}}  \geq  (m - q)\sum_{p = 1}^i y_p  - \sum_{ j \in  R_q} D_{ji},
		\end{align*}	
	where the first inequality follows from the fact that the set $\{\bar \sigma_{i(1)}, \bar \sigma_{i(2)}, \dots, \bar \sigma_{i(m -q)} \}$ is a possible choice of $ R_q$, and the second inequality follows from the definition of the  permutation $\bar \sigma$. Hence, the polynomial class of inequalities \eqref{eq:strong} implies all inequalities of the form \eqref{eq:weak}, which completes the proof.
	\end{proof}

	\begin{example} (Continued.) Let $i = 2$, then the value of  $\Theta'_2$ can be captured by the following $k + 1 = 3$ inequalities
		\begin{subequations}
			\begin{align}
				\label{eq:ex1}	&\Theta_2' \geq 5(y_1 + y_2) - 7 - 9 - 9 - 10 - 11, \\
				\label{eq:ex2}	&\Theta_2' \geq 4(y_1 + y_2) - 7 - 9 - 9 - 10, \\
				\label{eq:ex3}	&\Theta_2' \geq 3(y_1 + y_2) - 7 - 9 - 9.   
			\end{align}
		\end{subequations}
		In the optimal solution, if every scenario is satisfied at time period 2, then inequality \eqref{eq:ex1} captures the value of $\Theta'_2$, and the other two inequalities provide  lower bounds on $\Theta'_2$. Suppose that in the optimal solution, one scenario is violated in time period 2, then the violated scenario must be the scenario with the highest cumulative demand at time period 2. Hence, inequality \eqref{eq:ex2} captures the correct value of $\Theta'_2$, and inequalities \eqref{eq:ex1} and \eqref{eq:ex3} yield  lower bounds for $\Theta'_2$. 
	\end{example}

	\begin{remark}
		We show that the proposed formulation can also be applied to the general two-stage chance-constrained program with simple recourse, equiprobable scenarios and finite probability space.  
		
		Given a scenario set $\Omega = \{1, 2, \dots, m\}$, let $\vx$ be the vector of the first stage decision variables,  $\vc$ be its cost vector, and $X$ be its feasible region. In addition,  the following scenario-dependent constraint set:
		$$ A_j \vx \geq b_j   $$
		is enforced only when scenario $j \in \Omega$ is satisfied by the chance constraint, where $A_j$ and $b_j$ are random coefficient matrix of $\vx$ and right-hand side vector with appropriate dimensions, respectively.  In addition, the $d$-dimensional simple recourse function,  \cite[see, e.g.,][]{birgebook:97}  is defined as:
		$$\sum_{i = 1}^d h_i [   \vu_i^\top \vx -g_{ji}   ]_+, \quad  j \in \Omega  ,  $$
		where $g_{ji}$ is scenario-dependent parameter, for all $j \in \Omega$ and $ i \in [d]$, and $\vu_i$ is the coefficient vector of the recourse function for $i$-th dimension, for all $i \in [d]$. Let $h_i, i\in [d]$ be a penalty term for the excess $[   \vu_i^\top \vx -g_{ji}   ]_+$ in the second stage.

		Assume that each scenario is equally likely. The deterministic equivalent of a general two-stage chance-constrained program with simple recourse, equiprobable scenarios, and finite probability space is stated as follows:
		\begin{subequations}
			\label{eq:ggen}
			\begin{align}
				\label{eq:ggen0} \min\; & \vc^\top \vx  + \frac{1}{m}\sum_{j = 1}^m \sum_{i = 1}^d h_i [  \vu_i^\top \vx -g_{ji}    ]_+ \\
				\label{eq:ggen1} \text{s.t.} \;& A_j x +\bar M_j z_j  \geq b_j \\
				\label{eq:ggen2} & \sum_{j = 1}^m  z_j \leq  k \\
				\label{eq:ggen3}  & \vx \in X, \vz \in \mathbb B^m,
			\end{align}
		\end{subequations}
	\end{remark}
	where \eqref{eq:ggen1}-\eqref{eq:ggen2} enforce the chance constraint, and  $\bar M_j$ is sufficiently large to make \eqref{eq:ggen1} redundant when $z_j=1$.  Since we have to introduce $O(md)$ new variables and constraints to linearize the nonlinear term in the cost function \eqref{eq:ggen0}, the deterministic equivalent program \eqref{eq:ggen} is a large-scale mixed-integer program, which is very hard to solve. 
	
	Let $\bar \Theta_i$, for all $i \in [d]$, be the additional variable that captures the value of the recourse function for dimension $i$. In addition, let $\sigma'$ be the permutation of scenarios  such that:
	$ g_{ \sigma'_{i(1)}i} \leq g_{\sigma'_{i(2)}i}  \leq \cdots \leq g_{\sigma'_{i(m)}i}. $ In order to simplify notation, let $g_{\sigma'_{i(j)}i} = g_{\sigma'_{i(j)}} $, for all $j = 1, 2, \dots, m$. Hence, according to Proposition \ref{pro:8}, we can rewrite the deterministic equivalent formulation \eqref{eq:gen} as:
	\begin{align*}
		\min\; & \vc^\top \vx  + \frac{1}{m}  \sum_{i = 1}^d  h_i\bar \Theta_i\\
		\text{s.t.} \; & \eqref{eq:gen1} - \eqref{eq:gen3} \\
		& \bar \Theta_i \geq \sum_{j = 1}^{m - q}  (m - q) \vu^\top_i \vx-g_{\sigma'_{i(j)}} , & i \in [d],  q \in [k]^+, \\
		& \bar \Theta \in \mathbb R^d_+.
	\end{align*}
	Here we only require $d$ new variables and $O(dk)$ many new constraints. In this case, we can greatly reduce the number of variables and constraints in the deterministic equivalent formulation, because $k \ll m$, for small $\epsilon$.
	
	\section{Computational Experiments}\label{sec:comp}
	
	
	In this section, we summarize our  computational  experience with various classes of valid inequalities and our new formulation. All runs were executed on a Windows Server 2012 R2 Data Center with 2.40GHZ Intel(R) Xeon(R) CPU and 32.0 GB RAM. The algorithms tested in the computational experiment were implemented using C programming language, with Microsoft Visual Studio 2012 and CPLEX 12.6.   A time limit of one hour is set.

	In our experiments, we compare the proposed new formulation \eqref{eq:newformulation} against the deterministic equivalent formulation \eqref{eq:dep} and Benders decomposition algorithm (see Appendix \ref{sec:Benders}), with different choices of valid inequalities. The first class of valid  inequalities \eqref{eq:newvalid} and its special case of mixing inequalities \eqref{eq:mixing} are valid for  the deterministic equivalent formulation,  the Benders master problem and the new formulation \eqref{eq:newformulation}.    However, the  second class of valid inequalities given in Appendix \ref{sec:stock}   include the inventory variables, hence they only apply to the deterministic equivalent formulation.  
 In Tables  \ref{tab:3} and \ref{tab:4}, each  row  reports the average of  three instances.  We let $f_i$  and $c_i$ to be randomly generated  from a discrete uniform distribution over [50, 100], and [5, 10], respectively, for all $i \in N$.  In addition,  we generate the demand in each period randomly, where $d_{ji}$ follows discrete uniform distribution [10, 30], for all $i \in N$ and $j \in \Omega$.  

In Table \ref{tab:3}, the ``DEP \eqref{eq:mixing}, \eqref{eq:newvalid}, \eqref{eq:news};" "B. D. \& Ineq. \eqref{eq:mixing}-\eqref{eq:newvalid};" ``N.F. \& \eqref{eq:mixing};" and ``N.F. \& \eqref{eq:mixing}-\eqref{eq:newvalid}"  columns report the performance of the deterministic equivalent formulation with the additional strengthening from inequalities \eqref{eq:mixing}, \eqref{eq:newvalid}, and \eqref{eq:news}; Benders decomposition algorithm with valid inequalities \eqref{eq:mixing} and \eqref{eq:newvalid}; new formulation with valid mixing inequalities \eqref{eq:mixing}; and new formulation with mixing inequalities and the proposed inequalities \eqref{eq:newvalid},  respectively. The number of mixing inequalities that can be added to both formulations is limited to 150, and based on the results, this limit is hit by every instance. 
	The ``Time" column reports the average solution time in seconds for the instances that are solved to optimality within time limit,  and the  ``Gap'' column reports the average optimality gap for the instances that reach the time limit. The `` - '' sign under the ``Time '' column indicates that no instance is solved to optimality within time limit. The `` * " sign indicates that CPLEX is not able to solve the instance due to memory limit, and no feasible solution is obtained.   In addition, we only add the proposed inequalities at the root node level.

	\begin{table}[htb]
		\centering
		\caption{Computational results comparing different formulations.} 	
		\label{tab:3}		\begin{tabular}{|c|c|cc|cc|cc|cc|cc} 
			\hline
			\multicolumn{2}{|c|}{Instances}  &\multicolumn{2}{|c|}{DEP \eqref{eq:mixing}, \eqref{eq:newvalid}, \eqref{eq:news}}  &\multicolumn{2}{|c|}{B. D. \& Ineq. \eqref{eq:mixing}-\eqref{eq:newvalid}  }  &\multicolumn{2}{|c|}{ N.F. \& \eqref{eq:mixing} }  &\multicolumn{2}{|c|}{ N.F. \& \eqref{eq:mixing}-\eqref{eq:newvalid}  }    \\	
			\hline
			$(\epsilon, n)$ &	$ m  \;(10^3)$  & Time  & Gap (\%)    &Time & Gap (\%) & Time  & Gap (\%)  & Time   & Gap (\%) \\ 
			\hline
								
			\multirow{3}{*}{(0.01, 5)}	&	 10 & 277 & 0   &199  & 0  & 143 & 0  & 92 & 0\\
			&		 20   & * & * & 860 & 0 & 441 & 0  & 387 & 0  \\
			&		30  & *  & *& - & 0.47  & - & 0.12  & 3534 & 0.05 \\
			\hline			
			
			\multirow{3}{*}{(0.01, 10)}	&	 10 & * & *   &- & 0.16   & - & 1.11  & - & 1.38\\
						&		 20   & * & * & - & 2.17 & - & 0.94  & 3416 & 0.90  \\
						&		30  & *  & *& * & *  & - & 3.28   & - & 2.36 \\
						\hline	
			
					\multirow{3}{*}{(0.01, 30)}	&	 3  & * & * & 1028   &0    & 185 & 0  & 127 & 0\\
						&		 4   & * & * & 1794  & 6.71  &524& 0  & 397 & 0  \\
						&		5  & *  & * & 3324 & 14.66  & 1472 & 0 & 1334 & 0 \\
						\hline	
						\multirow{3}{*}{(0.01, 40)}	&	 3 & * & * & 1179   & 0  & 723 & 0   & 606 & 0\\
						&		 4  & * & *  & - & 23.02   &1864 & 0.24   & 1690 & 0.07  \\
						&		5   & *  & * & -  & 14.51   & 3321 & 0.71  & 2793 & 0.57 \\
						\hline

			\hline				

		\end{tabular}
		
	\end{table}
	
	\begin{table}[htb]
		\centering
		\caption{Additional information for the experiments in Table \ref{tab:3}.} 	
		\label{tab:4}		\begin{tabular}{|c|c|cc|cc|ccc|ccc} 
			\hline
			\multicolumn{2}{|c|}{Instances} &\multicolumn{2}{|c|}{B. D. \& Ineq. \eqref{eq:mixing}-\eqref{eq:newvalid}  }  &\multicolumn{2}{|c|}{ N.F. \& \eqref{eq:mixing} }   &\multicolumn{3}{|c|}{N.F. \& \eqref{eq:mixing}-\eqref{eq:newvalid} }    \\	
			\hline
			$(\epsilon, n)$ &	$ m  \;(10^3)$  & Nodes & Opt.Cut   & Nodes & R.Gap (\%) & Nodes  \;\;  & R. Gap (\%) \;\; & Cuts \\ 
			\hline
			\multirow{3}{*}{(0.01, 5)}& 10 &1828 & 42398 & 2028 & 1.08   &880  & 1.00  & 7\\
						&		 20 & 12493 & 121553& 737  & 3.90 & 623 & 3.43   & 6 \\
						&		30  & 48676 & 373193 & 54667 & 3.88 & 41379 & 3.41  & 7 \\
				\hline
								\multirow{3}{*}{(0.01, 10)}	&	 10  & 55031 & 98625& 36067 & 3.61   &33751  & 3.46  & 12\\
								&		 20 &28112 & 288242 & 29315  & 6.76  & 32529 & 4.96   & 10 \\
								&		30 & * & * & 6327 & 7.92 & 9088 & 6.78  & 10 \\
			\hline
			\multirow{3}{*}{(0.01, 30)}	&	 3  &5047 & 29828 & 837 & 3.60   & 359 & 2.88  & 23\\
			&		 4 & 9696 & 64784 & 2233 & 5.46  & 1892 & 3.59 & 26 \\
			&		5  & 12267 & 88599 & 5917 & 7.42& 5653 & 6.08 & 22 \\
			\hline	
			\multirow{3}{*}{(0.01, 40)}	&	 3 & 5026 & 31672 & 2088 & 3.88  & 1608 & 3.76 & 16\\
			&		 4 & 9397 & 103936 &6494 & 6.36 & 5729 & 2.53  & 19 \\
			&		5  & 8150 & 64154 & 7375 & 3.96  & 7672 & 3.24  & 18 \\
			\hline
		\end{tabular}
		
	\end{table}

	As we can see from Table \ref{tab:3}, the deterministic equivalent formulation cannot solve most of the instances, due to the memory limit. The Benders decomposition provides  slightly better results, since it is able to find a feasible solution. However, for the instances with 30 or 40 time periods,  the optimality gap of Benders decomposition algorithm is very large. The proposed new formulation provides a big improvement. It can solve most of the instances to optimality. For the instances that reach the time limit, the optimality gap is small. Finally, the effectiveness of the proposed inequalities \eqref{eq:newvalid} is shown in the last column. It provides the best results, with generally the smallest solution time and optimality gap. 
	
	In Table \ref{tab:4}, we report additional information on the average root gap (``R.Gap \%") and  number of nodes explored during the branch-and-bound process (``Nodes").  The column ``Opt.Cut" reports the number of optimality cuts added to  the Benders master problem.  The column ``Cuts" reports the number of the proposed inequalities \eqref{eq:newvalid} added to the new formulation  in addition to the mixing cuts \eqref{eq:mixing}, which are special cases of inequalities   \eqref{eq:newvalid}.  As we can see from Table \ref{tab:4},  because we only add the proposed inequalities \eqref{eq:newvalid} at the root node after adding the violated inequalities \eqref{eq:mixing}, the number of additional  inequalities \eqref{eq:newvalid}  is  not very large. However, the new cuts are   beneficial; the number of branch-and-bound nodes is  reduced with the proposed inequalities \eqref{eq:newvalid}, and the root node gap with the new inequalities is also smaller in most cases. As a result, more instances are solved to optimality within the time limit.  In addition, compared with the results from Benders decomposition, the proposed new formulation uses much fewer ``optimality cuts" to capture the second-stage inventory value. For example, for the instances where $m = 10000$ and $n = 5$, the Benders decomposition algorithm requires 42398 optimality cuts. In contrast,  the proposed new formulation only requires $m \times \epsilon \times n = 500$ additional inequalities to fully capture the second-stage inventory value.  As a result,  the proposed new formulation \eqref{eq:newformulation} provides a significant improvement in solution time.   
  
  We also tested the effectiveness of  adapting the  extended formulation of \cite{extended}  for  deterministic ULS to strengthen the deterministic equivalent of SPLS. However, we observe that it slows down the deterministic equivalent model further, so we do not report our computations with this formulation.

\section*{Acknowledgements}
The authors are supported, in part, by the National Science Foundation grant 1055668.

	\bibliographystyle{apalike}
	\bibliography{LS}

	\appendix
	
	\section{Valid inequalities that involve stock variables}  \label{sec:stock}
	In this section, we study the polyhedral structure of the deterministic equivalent formulation which includes the stock variables. Let 
	$		P_+  =   \{ {(\mathbf x, \mathbf y, \mathbf z, \vs}) \;|\; \eqref{eq:dep1} - \eqref{eq:dep6} \}.$
	
	\begin{proposition} 
		For $\ell = 2, \dots, n$,    let $T_\ell := \{t_{\ell(1)}, t_{\ell(2)}, \hdots , t_{\ell(a_\ell)}\}\subseteq T_{\ell}^*$, where 
		$D_{t_{\ell(1)}} \geq D_{t_{\ell(2)}} \geq \cdots \geq D_{t_{\ell(n)}} $. For $j\in \Omega$, the  inequalities
		\begin{align} \label{eq:news}
			s_{j(\ell-1)}  + &(D_{t_{\ell(1)}} - D_{j\ell-1})x_{\ell} + \sum_{p = 1}^{a_{\ell}}  (D_{t_{\ell(p)}} - D_{t_{\ell (p + 1) }}) z_{t_{\ell(p)}} \geq   D_{t_{\ell(1)}} - D_{j\ell-1},
		\end{align}
		are valid for $P_+$. 
	\end{proposition}
	\begin{proof}
		If $x_{\ell} = 1$, then inequality \eqref{eq:news} is trivially satisfied. Otherwise, $y_{\ell} = 0$. Because 
		$s_{j(\ell-1)}  \geq \sum_{p = 1}^{\ell - 1} y_p  - D_{j(\ell-1)}
		= \sum_{p = 1}^{\ell}  y_p  - D_{j(\ell-1)}$,
		the validity of inequality \eqref{eq:news} follows from the validity of the  mixing inequality \eqref{eq:mixing} for  time period $\ell$. 
	\end{proof}

	\renewcommand{\theexample}{\ref{ex:ex1}}
	\begin{example} (Continued.) Let $\ell = 2, j=1$ and $T_\ell = \{ 3, 4 \}$, then  we obtain:	
		\begin{align*}	
			&s_{11} + (D_{t_{2(1)}} - D_{t_{1(1)}}) x_2 +  (D_{t_{2(1)}} - D_{t_{2(2)}}) z_3 +  (D_{t_{2(2)}} - D_{t_{2(3)}}) z_4  \geq (D_{t_{2(1)}} - D_{t_{1(1)}}), 
		\end{align*}
		which is equivalent to:
		$$s_{11} + 5x_2  + z_3 + z_4 \geq 5.  $$
		In fact, this inequality is a facet-defining inequality for this problem, as we  show in Proposition \ref{pro:1}.
	\end{example}

	Next, we show the strength of the proposed inequalities \eqref{eq:news}.
	
	\begin{proposition}	\label{pro:1}
		For $\ell = 2, \dots, n$ and $T_\ell\subseteq T_\ell^*$, if $\sigma_{\ell - 1(1)} \not \in T^*_\ell \cup \{\sigma_{\ell (k + 1)}\}$, $j =\sigma_{\ell - 1 (1)}$  and $t_{\ell{(1)}} = \sigma_{\ell(1)}$, then inequality \eqref{eq:news} is facet-defining for $conv(P_+)$. 
	\end{proposition}

	\begin{proof} 
		First, we show that under the conditions stated in Proposition \ref{pro:1},  inequality \eqref{eq:news} is facet-defining for the convex hull of the  polyhedron:
		$		P_{s_{j(\ell - 1)}}  =   \{ ({\mathbf x, \mathbf y, \mathbf z, s_{j(\ell - 1)}}) \ \in B^n \times  \mathbb R_+^n \times B^m \times \mathbb R_+\;|\; 
		s_{j(\ell - 1)} \geq \sum_{p = 1}^{\ell - 1} y_p -  D_{j(\ell-1)}, \eqref{eq:dep1} - \eqref{eq:dep4}\},	
		$	
		in which we only consider the stock variable for  scenario $j = \sigma_{\ell - 1(1)}$ at time period $\ell -1$.	To show that inequality \eqref{eq:news} is facet-defining for $conv(P_{s_{j(\ell - 1)}}  )$, we need to find $dim(P_{s_{j(\ell - 1)}})= 2n + m$ affinely independent points $(\mathbf x, \mathbf y, \mathbf z, s_{j(\ell-1)})$ that satisfy inequality \eqref{eq:news} at equality. 
		
		Let $g(t_{i(p)})$, for all $t_{i(p)} \in T_i$ and $i+1 \in \bar S\cup\{\ell+1\}$,  be a unique mapping such that the scenario $t_{i(p)}$ has the $g(t_{i(p)})$-th largest cumulative demand at time period $i$.	We first consider the following set of feasible points: 
		\begin{align*}
			& (\mathbf e_1 + \ve_{\ell+ 1}, \mathbf{\bar y}_{p}^\ell + \ve_{\ell+ 1} M_{\ell + 1}, \sum_{i = 1}^{g(t_{\ell(p)}) - 1} \mathbf {e}_{\sigma_{\ell(i)}},    \bar y^\ell_{p1} -  D_{\sigma_{\ell - 1(1)}}), \quad p  \in [a_\ell + 1],
		\end{align*}
		where $\bar \vy^q_p$ is defined in the proof of Proposition \ref{prop:eq1facet}.	To see the feasibility of these points, note that if $\sigma_{\ell - 1(1)} \not \in T^*_\ell \cup \{\sigma_{\ell (k + 1)}\}$,  we must have $\bar y_{a_{\ell} + 1,1}^\ell -  D_{\sigma_{\ell - 1(1)}} \geq 0 $. Hence, we obtain $ a_\ell+1$ affinely independent points that satisfy inequality \eqref{eq:news} at equality. 
		
		Next,  consider the following set of points:
		\begin{align*}
			& (\mathbf e_1 + \ve_{\ell + 1}, \mathbf{\bar y}_{1}^\ell + \ve_{\ell+ 1} M_{\ell + 1}, \mathbf {e}_p , \bar y^\ell_{11} -  D_{\sigma_{\ell - 1(1)}}), \quad \forall p = \Omega \setminus  T_\ell .
		\end{align*}
		These $ m - a_\ell$ points are feasible, affinely independent from all other points, and satisfy inequality \eqref{eq:news} at equality.	
		Next, we consider the following set of points:
		\begin{align*}
			& (\mathbf e_1 + \ve_{\ell }, \mathbf{\bar y}_{1}^{\ell - 1} + \ve_{\ell} M_\ell, \mathbf 0, 0),   \\
			& (\mathbf e_1 + \ve_{\ell} + \ve_{p }, \mathbf{\bar y}_{1}^{\ell - 1 } + \ve_{\ell }  M_\ell, \mathbf 0,  0),  p \in N \setminus [\ell ],\\ 		
			& (\mathbf e_1 + \ve_{\ell } + \ve_{p }, \mathbf{\bar y}_{1}^{\ell - 1} + \ve_{\ell}  M_\ell + \ve_{p} \triangle, \mathbf 0, 0 ),  p \in N \setminus [\ell],
		\end{align*}
		where $ 0<\triangle < M_p$, for all $p \in N \setminus [\ell] $. 	These $2(n - \ell)+1$ points are feasible, affinely independent  from all other points, and satisfy inequality \eqref{eq:news} at equality.
		
		Next, we consider the following set of points:	
		\begin{align*}
			& (\mathbf e_1 + \ve_{\ell } + \ve_{\ell+ 1 }, \mathbf{\bar y}_{1}^{\ell - 1 } + \ve_{\ell} ( M_\ell - \triangle_1) + \ve_{\ell + 1} M_{\ell + 1}, \mathbf 0, 0),  \\		
			& (\mathbf e_1 + \ve_{p } + \ve_{\ell }, \mathbf{\bar y}_{1}^{\ell -1 } + \ve_{\ell} M_\ell , \mathbf 0, 0 ),  p \in [\ell - 1] \setminus \{1\},\\	
			& (\mathbf e_1 + \ve_{p } + \ve_{\ell }, \mathbf{\bar y}_{1}^{\ell - 1  } + \ve_{\ell} M_\ell + \triangle_2 (\ve_p - \ve_1) , \mathbf 0, 0), p \in [\ell - 1] \setminus \{1\},
		\end{align*}
		where $ 0<\triangle_1 \leq \bar y^{\ell - 1}_{11}$, and $0<\triangle_2\le \min\{ \bar y^{\ell - 1}_{11} - D_{\sigma_{1(1)}},M_p\}$, for all $p \in [\ell - 1] \setminus \{1\}$.  It is easy to see that these $2\ell - 3$ points are feasible, affinely independent from other points, and satisfy inequality \eqref{eq:news} at equality. 
		Finally, consider  the  feasible  point:
		$			 (\mathbf e_1 + \ve_{\ell } , \mathbf{y}^{*} + \ve_{\ell} M_\ell, \ve_{\sigma_{\ell - 1(1)}}, 0),$	
		where $y^{*}_1 = D_{\sigma_{\ell - 1(2)}}$, and $y^{*}_i = 0$, for all $i = 2, \dots, n$.   
		\exclude{
			It is easy to see that this point is feasible \comment{Do we have to add some sentences to show why this point is feasible? The idea is that we let $y^*_1$ equals to the second largest scenario in period $\ell - 1$, if $y^*_1 $ is less than the largest scenario in period $\ell - 2$, then the largest scenario in time period $\ell - 2$ must also be the largest scenario in period $\ell - 1$, because $y^*_1$ equals to the second largest scenario in period $\ell - 1$. since we let $z_{\sigma_{\ell - 1}(1)} = 1$, then this $y^*_1$ only has to be greater than or equal to the second largest scenario in the time period $\ell - 2$, which is true, since demand in each time period is nonnegative, the second largest scenario in the time period $\ell - 1$ must be greater than or equal to the second largest scenario in time period $\ell - 2$.   }, }
		This point is affinely independent from all other points, and satisfies inequality \eqref{eq:news} at equality.  Hence, we have  $2n + m$ affinely independent points that satisfy inequality \eqref{eq:news} at equality, which shows that the proposed inequality is facet-defining for $conv(	P_{s_{j(\ell - 1)}} )$.
		
		To show that the proposed inequality is also facet-defining for  $conv(P_+)$, let:
		$ (\tilde \vx^p, \tilde \vy^p, \tilde \vz^p , \tilde s^p_{j(\ell-1)}),  p \in  [ 2n + m ],$
		be the affinely independent points constructed for $conv(P_{s_{j(\ell - 1)}}  )$. Then, we construct the  set of points:
		$ (\tilde \vx^p, \tilde \vy^p, \tilde \vz^p , \tilde \vs^p ),  p \in  [ 2n + m ],$
		where $\tilde s^p_{qi} = \max \{\sum_{u = 1}^i \tilde y^p_u - D_{qi} \;, \; 0 \} $ for $q\in \Omega, i\in N$. These ``extended'' points are feasible, affinely independent, and satisfy  inequality \eqref{eq:news} at equality. 		
		Finally, for each inventory variable $s_{pi}$ such that $p \not = \sigma_{\ell - 1 (1)}$ or $i \not = \ell - 1$, we construct the  set of points:
		$		(\tilde \vx^1, \tilde \vy^1, \tilde \vz^1 , \tilde \vs^1 ) + (\mathbf 0, \mathbf 0, \mathbf 0, \triangle \ve_{pi}),  p \not = \sigma_{\ell - 1 (1)}, i \not = \ell - 1,$
		where $\triangle > 0$, and  $\ve_{pi}$ is an $m\times n$ dimensional matrix such that the  $(p,i)$-th entry  equals 1, and all other entries are 0.  These $nm - 1$ points are feasible, affinely independent from other points, and satisfy  inequality \eqref{eq:news} at equality.  
		Hence, we obtain $2n + m + mn - 1$  feasible, affinely independent points that satisfy  inequality \eqref{eq:news} at equality, which completes the proof. \qed
	\end{proof}

	\noindent \textbf{Separation of inequalities \eqref{eq:news}}: 
	Given a fractional solution of the deterministic equivalent formulation $(\hat\vx, \hat\vy, \hat\vz, \hat\vs)$, we solve the problem \eqref{eq:sep2} to obtain $\hat Y_i, i\in N\setminus \{1\}$.
	\exclude{
		\begin{align} \label{eq:seps}
			W_\ell = \; \min_{T_\ell \subseteq T_\ell^*} \{ \; (D_{t_{\ell(1)}} - D_{j\ell - 1})\hat x_{\ell} + \sum_{p = 1}^{a_{\ell}}  (D_{t_{\ell}(p)} - D_{t_{\ell (p + 1) }}) \hat z_{t_{\ell(p)}} \}, 
		\end{align}
		which is identical to the separation of the mixing inequalities,  which takes $O(nk\log(k))$ time \cite[]{Gunluk}. 
	}
	Then, with a linear pass, we   add the violated inequality \eqref{eq:news} for $\ell\in N\setminus \{1\}, j=\sigma_{\ell-1(1)},$ if  $\hat s_{j(\ell-1)}  +(D_{t_{\ell(1)}} - D_{j(\ell - 1)})\hat x_{\ell} + \hat Y_\ell< D_{t_{\ell(1)}} - D_{j(\ell - 1)}$. Otherwise, there is no violated inequalities \eqref{eq:news}. The overall running time is $O(nk\log(k))$. In addition, since we consider a single time period at a time, the separation procedure is exact.

	\section{A Benders decomposition algorithm} \label{sec:Benders}
	
	There are $m  n$  stock variables, which could cause computational difficulty as the size of the problem increases. In this section, we study a Benders decomposition algorithm. Let $\theta_j$, for all $j \in \Omega$, represent the additional variable that captures the second-stage cost of scenario $j$. The relaxed master problem is
		\begin{align*}
			\textbf{MASTER:} \;	   	\min \; &\vf^\top \vx + \vc^\top \vy + \frac{1}{m} \sum_{j = 1}^m \pi_j  \theta_j \\
				\text{s.t.}\; & \sum_{i = 1}^t y_i \geq \sum_{i = 1}^t d_{ji}(1 - z_j), \; & t \in N, j \in \Omega \\
				&\sum_{j = 1}^m  z_j \leq k \\
					&y_i \leq M_i x_i, & i \in N \\
					&\theta\in \mathbb R^m_+ , \vy  \in \mathbb R_+^n,\vx \in \mathbb B^n,  \vz \in \mathbb B^m,
		\end{align*}	
	where we relax the constraint \eqref{eq:dep5} which captures the second-stage cost of each scenario. Note that  the first class of proposed inequalities \eqref{eq:newvalid} is valid for the master problem. However, since the second class of valid inequalities \eqref{eq:news} involves the stock variables, it cannot be directly applied to the master problem. For each $j \in \Omega$, the subproblem is stated as:
		\begin{align*}
			\theta_j= \;	\min \; &  \vh^\top \vs_j  \\
			& s_{ji} \geq \sum_{i = 1}^t(y_i - d_{ji}), \quad    t \in N & (\gamma_{ji})\\
			&\vs  \in \mathbb R_+^n,
		\end{align*}
	where $\gamma_{ji}$ is the dual variable associated with $i$-th time period of $j$-th scenario. Next, the corresponding dual variable for $j$-th scenario is stated as:
	\begin{subequations}
		\label{eq:duals}
		\begin{align}
			\theta_j \geq \;& 	\max \sum_{i = 1}^n \Big(\sum_{t = 1}^i(y_t - d_{jt}) \Big) \gamma_{ji}    \\
			\label{eq:d1}& \gamma_{ji} \leq h_i, & i \in N ,\\
			\label{eq:d2}& \mathbf\gamma_j \in \mathbb R_+^{n} .
		\end{align}
	\end{subequations}
	Note that according to \cite{merve}, we can apply the second class of  valid inequalities \eqref{eq:news} to the subproblems, to further strengthen the quality of the Benders optimality cuts added to the master problem. However, this implementation  did not lead to improvements in solution time for our test instances, hence we do not report experiments with this version of Benders in our computational study in Section \ref{sec:comp}. 
	
	Given a first stage solution $(\hat y, \hat \theta)$, instead of solving the dual problem \eqref{eq:duals} as a linear problem, we can take advantage of the special structure of \eqref{eq:duals} and generate Benders optimality cuts in $O(n)$ time: for each $i \in N$, if the term $\sum_{t = 1}^i(y_t - d_{jt}) < 0$, then $\gamma_{ji} = 0$, because of the nonnegativity of $h_i$, for all $i \in N$. Otherwise, $\gamma_{ji} = h_i$. Let $\gamma^*_j$ be the optimal dual solution for scenario $j$, if $\hat \theta_j < \sum_{i = 1}^n \Big(\sum_{t = 1}^i(y_t - d_{jt}) \Big) \gamma^*_{ji} $, then we add the following optimality cut to the master problem:
	\begin{align*}
		\theta_j \geq \sum_{i = 1}^n \Big(\sum_{t = 1}^i(y_t - d_{jt}) \Big) \gamma^*_{ji} ,
	\end{align*} 
	to cut off the suboptimal solution.
	
	However,  although we can solve the subproblem in $O(n)$ time, there are an exponential number of possible Benders optimality cuts for each scenario. As  shown in our computational study in  Section \ref{sec:comp}, as the number of time periods ($n$) grows, the Benders decomposition algorithm becomes ineffective.

\end{document}